\documentclass[10pt,reqno]{amsart}

 \usepackage{amsmath,amsfonts,amssymb,amscd,amsthm,amsbsy,bbm, epsf,calc,enumerate,color,graphicx,verbatim,cite,import}
\usepackage{color}
\usepackage{datetime,appendix}
\usepackage{chapterbib,epstopdf}

\usepackage{float}

\usepackage{graphicx}

\usepackage{ifpdf}
\ifpdf
    \usepackage[colorlinks,citecolor=black,linkcolor=black]{hyperref}
\fi

\usepackage[mathlines]{lineno}

\theoremstyle{plain}
\newtheorem{definition}{Definition}[section]
\newtheorem{thm}{Theorem}[section]
\newtheorem{theorem}{Theorem}[section]

\newtheorem{lemma}[thm]{Lemma}

\newtheorem{corollary}[thm]{Corollary}

\theoremstyle{definition}

\theoremstyle{remark}                  

\definecolor{darkgreen}{rgb}{0,0.4,0}

\numberwithin{equation}{section}

\def\tr{\textnormal{tr}}

\def\XXint#1#2#3{{\setbox0=\hbox{$#1{#2#3}{\int}$ }
\vcenter{\hbox{$#2#3$ }}\kern-.6\wd0}}

\title[Integro-PDEs]{Regularity theory for Second order Integro-PDEs}
 
\author[Chenchen Mou and Yuming Paul Zhang]{\bfseries Chenchen Mou and  Yuming Paul Zhang}

\address{
(C. Mou) Department of Mathematics \\ 
University of California   \\ 
Los Angeles\\
USA}
\email{muchenchen@math.ucla.edu}

\address{
(Y. Zhang) Department of Mathematics \\ 
University of California   \\ 
Los Angeles\\
USA}
\email{yzhangpaul@math.ucla.edu}

\begin{document}

\vspace{18mm} \setcounter{page}{1} \thispagestyle{empty}

\begin{abstract}
This paper is concerned with higher H\"older regularity for viscosity solutions to non-translation invariant second order integro-PDEs, compared to \cite{mou2018}. We first obtain $C^{1,\alpha}$ regularity estimates for fully nonlinear integro-PDEs. We then prove the Schauder estimates for solutions if the equation is convex.
\end{abstract}

\maketitle

\vspace{.1cm}

\vspace{.1cm}

\bigskip

\section{Introduction}

We consider the following Hamilton-Jacobi-Bellman-Isaacs (HJBI) integro-PDE
\begin{equation}\label{main}
    \mathcal{I}[x,u]:=\sup_{\alpha\in\mathcal{A}}\inf_{\beta\in\mathcal{B}}\left\{-\tr{ \, a_{\alpha\beta}(x) D^2 u(x)}-I_{\alpha\beta}[x,u]+b_{\alpha\beta}(x)\cdot Du(x)+c_{\alpha\beta}(x)u(x)+f_{\alpha\beta}(x)\right\}=0
\end{equation}
in a bounded domain $\Omega\subset \mathbb{R}^d
$ with the following boundary data
\begin{equation}
    \label{main boundary}
    u(x)=g(x) \quad\text{in }\Omega^c.
\end{equation}
The coefficients $a_{\alpha\beta}, b_{\alpha\beta}, c_{\alpha\beta}, f_{\alpha\beta}$ are uniformly continuous in $\Omega$, uniformly in $\alpha\in \mathcal{A}$ and $\beta\in\mathcal{B}$. The L\'{e}vy measure has the form
\begin{equation}\label{eq:nonlocal operator}
    I_{\alpha\beta}[x,u]=\int_{\mathbb{R}^d\setminus \{0\}}( u(x+z)-u(x)-\chi_{B_1}(z)Du(x)\cdot z)N_{\alpha\beta}(x,z)dz,
\end{equation}
where $N_{\alpha\beta}(\cdot,z)$ is uniformly continuous in $\Omega$ for each fixed $z\in\mathbb R^d\setminus\{0\}$, uniformly in $\alpha\in\mathcal{A}$ and $\beta\in\mathcal{B}$.
We require $\lambda I\leq a_{\alpha\beta}\leq \Lambda I$ in $\Omega$ for some $0<\lambda\leq\Lambda$ and
\begin{equation} \label{cond K}
    0\leq N_{\alpha\beta}(x,\cdot)\leq K(\cdot)\quad\text{in $\mathbb R^d\setminus\{0\}$}
\end{equation}
for some measurable function $K$ on $\mathbb R^d\setminus\{0\}$ satisfying
\begin{equation}\label{integ K}
\int_{\mathbb{R}^d\setminus\{0\}}\min\{|z|^2,1\}K(z)dz<+\infty.
\end{equation}
In this paper we extend $K$ and $N_{\alpha\beta}(x,\cdot)$ for any $\alpha\in\mathcal{A}$, $\beta\in\mathcal{B}$, $x\in\Omega$ to functions on $\mathbb R^d$ by setting $N_{\alpha\beta}(x,0)=K(0)=0$.

\bigskip

Here we want to emphasize that the assumption \eqref{cond K} we impose on $N_{\alpha\beta}$ does not guarantee that the nonlocal operator $I_{\alpha\beta}$ has an order i.e. there exists a constant $\sigma$ such that $I_{\alpha\beta}[x,u(r\cdot)]=r^{\sigma}I_{\alpha\beta}[rx,u(\cdot)]$. Moreover, we notice that, with \eqref{integ K}, the nonlocal operator $I_{\alpha\beta}$ behaves like a second order operator and is not well defined acting on any unbounded function. These features of the nonlocal term lead to essential difficulties, which we will specify them later, in the proof of the regularity results.

\medskip

In this paper we study $C^{1,\alpha}$ and $C^{2,\alpha}$ regularity results for second order uniformly elliptic integro-PDEs as a consequence of the uniform ellipticity of the differential operators. The motivation of studying such regularity results comes from the stochastic optimal control problems. Two mainly concerned problems in the stochastic optimal control theory are the stochastic representation formulas for HJB equations and the optimal policies. 
 Indeed, to obtain the stochastic representation for a degenerate HJB integro-PDE, we could first derive it for the approximating equation, by adding $\epsilon\Delta u$ to the HJB integro-PDE, which is a uniformly elliptic equation where the uniform ellipticity comes from the second order term. The $C^{2,\alpha}$ regularity for such uniformly elliptic HJB integro-PDE is crucial for the application of the It\^o formula for general L\'evy processes to derive the representation for its solution and the optimal controls. Then, by an approximation (``vanishing viscosity") argument, we can obtain the representation for the approximating equation converges to that for the degenerate HJB integro-PDE in $L^\infty$ sense as $\epsilon\to 0$. And thus the stochastic representation for the degenerate HJB integro-PDE and a $\epsilon$-optimal control are obtained. We refer to \cite{FS,KS10,GMS} for more on the application of regularity theory to the stochastic optimal control problems. The other motivation of studying regularity for the integro-differential operator $\mathcal{I}$ in \eqref{main} comes from the generality of the operator. Indeed it has been proved that if $\mathcal{I}$ maps $C^2$ functions to $C^0$ functions and satisfies the degenerate ellipticity assumption then $\mathcal{I}$ should have the form in \eqref{main}, see \cite{Cour,ns}. 

\medskip

Existence of a $C^{2,\alpha}$ solution of the Dirichlet boundary value problem for a uniformly parabolic convex second order integro-PDE has been obtained by R. Mikulyavichyus and G. Pragarauskas in \cite{MP4} under an additional assumption on the nonlocal term. The convexity assumption implies the equation studied there are the convex and parabolic version of \eqref{main}. Then the assumption on the nonlocal term in \cite{MP4} translated to our equation is that for every $\alpha\in\mathcal{A}$, $\beta\in\mathcal{B}$, $z\in \mathbb R^d$ and $x\in \Omega$, the kernel $N_a(x,z)=0$ if $x+z\not\in \Omega$. For the associated optimal control problem this corresponds to the requirement that the controlled diffusions never exit $\bar\Omega$ and thus the boundary condition is different from the one in \eqref{main boundary}. With a similar assumption, they also studied in \cite{MP5,MP6} existence of viscosity solutions, which are Lipschitz in $x$ and $1/2$ H\"older in $t$, of Dirichlet and Neumann boundary value problems for parabolic degenerate Hamilton-Jacobi-Bellman (HJB) integro-PDEs where the nonlocal operators are of L\'evy-It\^o form. In \cite{M}, one of the authors studied semiconcavity of viscosity solutions for degenerate elliptic HJB integro-PDEs. If the control set is finite, existence of $C^{2,\alpha}$ solutions of Dirichlet boundary value problems for uniformly parabolic HJB integro-PDEs with nonlocal terms of L\'evy-It\^o type was investigated in \cite{mou2017}. At the end we mention that there are many recent $C^{1,\alpha}$ and $C^{2,\alpha}$ regularity results for purely nonlocal equations, see e.g. \cite{LL2,LL3,CD1,CD2,CK,TDZ, TZ, TJ1, TJ,Kri, Se, L1,HY}, where regularity is derived as a consequence of uniform ellipticity/parabolicity of the nonlocal part.


\medskip

The first goal of the paper is to show the $C^{1,\alpha}$ regularity estimates for \eqref{main}. We adapt the approach from \cite{LL2} using the blow up and approximation techniques. We first rescale the solution $u$ of \eqref{main} and consider
\[v(x):=u(r(x-x_0)+x_0) \text{ for each }x_0\in\Omega\text{ and $r>0$}. \]
Since the nonlocal operator $I_{\alpha\beta}$ is not scaling invariant, $v$ satisfies a different equation $\mathcal{I}^r(x_0)[x,v]=0$, see \eqref{eqL:Ir}. As \cite{LL2}, we need to find an operator, which has $C^{1,\alpha}$ regularity, such that $\mathcal{I}^r(x_0)$ is close to it with respect to some weak topology for small $r$. Fortunately, with \eqref{cond K} and \eqref{integ K} we are able to prove that $\mathcal{I}^r(x_0)$ converges to a uniformly elliptic local operator as $r\to 0$, see Lemma \ref{lem:weak con}. {In the Lemma, we do not blow up the operator at one point in $\Omega$, instead we consider operators of the form $\{\mathcal{I}^{r_k}(x_k)\}_k$ with $r_k\to 0$. This leads to the weak convergence of $\{\mathcal{I}^r(x_0)\}_r$ independent of $x_0$.} We then apply Lemma \ref{lem:weak con} to prove an approximation Lemma, see Lemma \ref{compare}, and the following $C^{1,\alpha}$ regularity. We state in an informal way here and will give the full result in Theorem \ref{thm 3.1}.
\begin{theorem}
If $u$ is a viscosity solution of \eqref{main} in $B_2$, then there exist constants $0<\alpha<1$ and $C>0$ such that
\begin{equation*}
\|u \|_{C^{1,\alpha}(B_{1})}\leq C\left(\|u\|_{L^\infty(\mathbb R^d)}+\sup_{\alpha\in\mathcal{A},\beta\in\mathcal{B}}\|f_{\alpha\beta}\|_{L^\infty(B_2)}\right).
\end{equation*}
\end{theorem}


The idea of the proof is similar to Theorem 52\cite{LL2}. However, because of the non-scale invariant nature of $\mathcal{I}$ and the weak integrability assumption \eqref{integ K}, there are several difficulties needed to be addressed. On one hand, only with \eqref{integ K}, any unbounded function can not be a test function of the nonlocal operator $I_{\alpha\beta}$. Unlike Theorem 52\cite{LL2}, we need to consider inductively \[w_k:=\left(u-\xi l_k\right)\left(\left(\nu^k(x-x_0)+x_0\right)\right)/\nu^{k(1+\alpha)}\] where $\xi\in C_c^\infty(\mathbb R^d)$, $l_k$ is a first order Taylor expansion of the approximation of $u$ in each scale at $x_0$ and $\nu\in (0,1)$ is a sufficiently small constant. Another challenge, caused by \eqref{integ K}, is that, to apply the approximation Lemma on $\{w_k\}_k$, we need $\{w_k\}_k$ to be uniformly bounded. This is not the case since we keep zooming in on $\xi$. So we need to further cut off $w_k$ uniformly by $\max\{-1,\min\{1,w_k\}\}$. And thus we need to take care of the errors generated by the cut-offs. On the other hand, $\{w_k\}_k$ is needed to be compact to apply the approximation Lemma. Because $I_{\alpha\beta}$ is not scale invariant, $w_k$ satisfies different maximal equations for each $k$. Thanks to Theorem 4.1\cite{mou2018}, the compactness of $\{w_k\}_k$ can be obtained by a uniform interior $C^\alpha$ estimate for $\{w_k\}_k$. We want to mention that, in \cite{TJ1, TJ}, Tianling Jin and Jingang Xiong also used an approximation argument to derive Schauder estimates for non-translation invariant purely nonlocal linear and fully nonlinear equations.


\medskip

After establishing the $C^{1,\alpha}$ estimates, in Section 5 we seek for $C^{2,\alpha}$ estimates for convex integro-PDEs i.e. \eqref{main} where $\mathcal{B}$ is singleton. We do not follow the ideas in \cite{CC,LL1} to first prove $C^{1,1}$ regularity estimates, and then combine the weak $L^\epsilon$ estimate and the oscillation lemma to obtain $C^{2,\alpha}$ regularity estimates. 
{A small drawback of this method is that, more regularity assumptions on the kernel are required.} {Another method given in \cite{Se} mainly makes use of a Liouville theorem and a blow-up procedure, where it is important that the kernel has an order which is preserved under some rescaling and the rescaled solutions have the same growth as the original one. However for us, to apply their method, more efforts are required to study the rescaled equations, and the corresponding solutions which might have growing $L^\infty$ bounds as the scale become small.} It is because of the non-scale invariant nature of $\mathcal{I}$ and the weak integrability assumption \eqref{integ K}. {We find a simpler way to solve this problem which is the method developed by the first author in \cite{mou2017}.} We first prove the a priori $C^{2,\alpha}$ estimate, which is the second main result in the manuscript. That is
\begin{theorem}
If $u\in C^2(B_2)\cap C^\alpha(\mathbb R^d)$ solves $\eqref{main}$, where $\mathcal{B}$ is singleton, classically in $B_2$, then
 \begin{equation*}
 \|u\|_{C^{2,\alpha}(B_1)}\leq C\left(\|u\|_{C^{\alpha}(\mathbb R^d)}+\sup_{\alpha\in\mathcal{A}}\|f_{\alpha}\|_{C^\alpha(B_2)}\right).
 \end{equation*}
\end{theorem}
See Theorem \ref{a priori estimate} for the full result. In the proof, we use the $C^{1,\alpha}$ estimate given in Section 4 to show that the nonlocal term is indeed a small perturbation compared to the second order term in \eqref{main}. With the a priori estimate, we then take a sequence of functions $\{u_q\}_q$ and each solves an integro-PDE, in the viscosity sense, where we truncate the kernels of the nonlocal term. We notice that the comparison principle for such truncated integro-PDE is not necessary to hold. However, the existence of such sequence is guaranteed by Perron's method in \cite{mou2018}. Without loss of generality, we can assume the solutions $\{u_q\}_q$ are classical solutions since we can consider one part of the nonlocal term as a first order term and the other part as a zero order term, and apply the classical $C^{2,\alpha}$ estimates. Using the a priori $C^{2,\alpha}$ estimate, we obtain a uniform $C^{2,\alpha}$ estimate for all $u_q$ and thus we can pass the limit to obtain a classical solution $u$ and its corresponding $C^{2,\alpha}$ estimate. At the end, by comparing viscosity solutions with the classical solution we are able to show that the classical solution $u$ we construct is the unique solution among all viscosity solutions.

\section{Notation and Preliminaries}

Throughout the paper, we assume that $\Omega$ is a bounded domain in $\mathbb R^d$. By $B_r(x)$ we denote a ball in $\mathbb{R}^d$ centered at $x$ with radius $r$. We write $B_r:=B_r(0)$ for abbreviation.

Let $O$ be a domain in $\mathbb{R}^d$. We say that $O$ satisfies the \textit{uniform exterior ball condition} if there exists $r_*>0$ such that for any $x\in \partial O$, there is a ball $B_{r_*}(y)\subset  O^c$ and $\{x\}=\partial B_{r_*} (y)\cap\partial O$. 

For any $U,V\subset\mathbb{R}^d$, define the distance between the two sets as
\[d(U,V):=\inf\{|x-y|,x\in U,y\in V\}.\]
For $\delta>0$ we denote by
\[O_\delta:=\{x\in O, d(x,O^c)\geq \delta\}.\]

\medskip

For $n\in\mathbb{N}$ and $\alpha\in (0,1)$, we denote by $C^{n,\alpha}(O)$ the set of functions defined in $\bar{O}$ whose $n$th derivatives are uniformly $\alpha$-H\"{o}lder continuous in $\bar{O}$. By $C^{n,\alpha}_{\rm loc}(O)$ we denote all functions whose $n$th derivatives are locally $\alpha$-H\"{o}lder continuous in $O$. We will briefly write $C^{\alpha}(O):=C^{0,\alpha}(O)$ and  $C^{\alpha}_{\rm loc}(O):=C^{0,\alpha}_{\rm loc}(O)$.
For any $u\in
C^{n,\alpha}(O)$, we define
\[
[u]_{n, \alpha; O}:=\left\{\begin{array}{ll}
\sup_{x\in O,|j|=n}|\partial^{j}u(x)|,&\hbox{if}\, \alpha=0;\\
\sup_{x, y\in O, x\not
=y,|j|=n}\frac{|\partial^{j}u(x)-\partial^{j}u(y)|}{|x-y|^{\alpha}},
&\hbox{if}\, \alpha>0,\end{array}\right.
\]
and
\[
\|u\|_{C^{n, \alpha}( O)}:=\left\{\begin{array}{ll}
\sum_{j=0}^{n}[u]_{j, 0, O}, &\hbox{if}\, \alpha=0;\\
\|u\|_{C^{n,0}( O)}+[u]_{n, \alpha; O}, &\hbox{if}\,
\alpha>0.\end{array}\right.
\]
We write ${\rm USC}(\mathbb{R}^d)$ (${\rm LSC}(\mathbb{R}^d)$) for the space of \textit{upper (lower) semicontinuous} functions in $\mathbb{R}^d$.
For any function $u$, we denote $u^+:=\max\{u,0\}$ and $u^-:=-\min\{u,0\}$.
We denote by $\mathcal{M}^d$ the class of $d\times d$ symmetric matrices with real entries.

\medskip

Let us introduce the rescaled operator $\mathcal{I}^r$ corresponding to $\mathcal{I}$:
\begin{equation}\label{eq:rescale fulloperator}
    \mathcal{I}^r(x_0) [x,u]:=r^2\mathcal{I}[r(x-x_0)+x_0,u(r^{-1}(\cdot-x_0)+x_0)]. 
\end{equation}
We define
\begin{equation}\label{eq:rescale nonlocal operator}
    I_{\alpha\beta}^r(x_0)[x,u]=\int_{\mathbb{R}^d}( u(x+z)-u(x)-\chi_{B_\frac{1}{r}}(z)Du(x)\cdot z)N^r_{\alpha\beta}(x_0,x,z)dz
\end{equation}
where
\[N_{\alpha\beta}^r(x_0,x,z)=r^{d+2}N_{\alpha\beta}(r(x-x_0)+x_0,rz).\]

Then we have
\[0\leq N^r_{\alpha\beta}(x_0,x,z)\leq K^r(z)\]
where $K^r(z):= r^{d+2}K(rz)$ satisfies
\[\int_{\mathbb{R}^d}\min\{|z|^2,1/r^2\}K^r(z)dz<+\infty.
\]
Using \eqref{eq:rescale fulloperator} and \eqref{eq:rescale nonlocal operator}
\begin{align}\label{eqL:Ir}
    &\mathcal{I}^r(x_0)[x,u]=\sup_{\alpha\in\mathcal{A}}\inf_{\beta\in\mathcal{B}}\left\{-\tr{ \, a_{\alpha\beta}(r(x-x_0)+x_0) D^2 u(x)}-I^r_{\alpha\beta}(x_0)[x,u]\right.\nonumber\\
    &\quad\left.+r\,b_{\alpha\beta}(r(x-x_0)+x_0)\cdot Du(x)+r^2\,c_{\alpha\beta}(r(x-x_0)+x_0)u(x)+r^2\,f_{\alpha\beta}(r(x-x_0)+x_0)\right\}.
\end{align} 
For abbreviation, we write
\[\mathcal{I}^ru(x):=\mathcal{I}^r(0)[x,u].\]

For $\Lambda>\lambda>0$ and any kernel $K$ satisfying \eqref{integ K},
we define the extremal operators for the second order term and the nonlocal term:
\begin{align*}
&    \mathcal{P}^+(D^2 u)(x):=\sup\{tr(A D^2u(x)),\, A\in \mathcal{M}^d, \lambda I\leq A\leq \Lambda I\},\\
&    \mathcal{P}^-(D^2 u)(x):=\inf\{tr(A D^2u(x)),\, A\in \mathcal{M}^d, \lambda I\leq A\leq \Lambda I\},
\end{align*}
\begin{align*}
&    \mathcal{P}_{K,r}^+( u)(x):=\int_{\mathbb{R}^d}\left(u(x+z)-u(x)-1_{B_{1/r}}(z)Du(x)\cdot z\right)^+ K^r(z)dz,\\
&    \mathcal{P}_{K,r}^-(u)(x):=\int_{\mathbb{R}^d}\left(u(x+z)-u(x)-1_{B_{1/r}}(z)Du(x)\cdot z\right)^- K^r(z)dz.
\end{align*}

\medskip

We denote \textit{half relaxed limits} of a sequence of functions $\{u_n\}_n$ by:
\[u^*(x)=\limsup_{k\rightarrow \infty}\left\{ u_n(x'):\,n\geq k,\, x'\in B_{1/k}(x)\right\}\]
and
\[u_*(x)=\liminf_{k\rightarrow \infty}\left\{ u_n(x'):\,n\geq k,\, x'\in B_{1/k}(x)\right\}.\]
Then we say $u_n$ \textit{$\Gamma$-converge} to $u$ if $u^*(x)=u_*(x)=u$.

\medskip

For the operator $\mathcal{I}$, we make the following assumptions: in $\overline{\Omega}$
\begin{itemize}
    \item[]  (A) $a_{\alpha\beta}(\cdot),b_{\alpha\beta}(\cdot),c_{\alpha\beta}(\cdot), f_{\alpha\beta}(\cdot)$ are uniformly continuous and bounded,\\
    
    \item[] (B) $\lambda I\leq a_{\alpha\beta}(\cdot)\leq \Lambda I$ for some $0<\lambda\leq\Lambda$,\\
    
    \item[] (C) for each $\sigma>0$, $N_{\alpha\beta}(\cdot,z)$ is uniformly continuous for each $z\in B_\sigma^c$ and \eqref{cond K} holds,
\end{itemize}
all uniformly in $\alpha\in A,\beta\in  B$.

In the last section we will further make use of the following conditions: in $\overline{\Omega}$
\begin{itemize}
    \item[]  (A') $a_{\alpha\beta}(\cdot),b_{\alpha\beta}(\cdot),c_{\alpha\beta}(\cdot), f_{\alpha\beta}(\cdot)$ are uniformly ${\alpha}$-H\"{o}lder continuous and bounded,\\
    
    \item[] (C') $\int_{\mathbb{R}^d}\min\{1,|z|^2\}N_{\alpha\beta}(\cdot,z)dz$ is uniformly ${\alpha}$-H\"{o}lder continuous and (C) holds,
\end{itemize}
all uniformly in $\alpha\in A,\beta\in  B$.

\medskip

By \textit{universal constants} we mean constants depending on $\lambda$, $\Lambda$, $d$, $K$, $b_{\alpha\beta}$ and $c_{\alpha\beta}$. We will denote universal constants by $C$ in the paper, possibly changing from one estimate to another.

\medskip

Now we give the definition of viscosity solutions to \eqref{main}.
\begin{definition}
A bounded function $u\in {\rm USC}(\mathbb{R}^d)$ is a viscosity subsolution of \eqref{main} if whenever $u-\varphi$ has a maximum over $\mathbb{R}^d$ at $x\in\Omega$ for some bounded test function $\varphi\in C^2(\Omega)\cap C(\mathbb{R}^d)$, then
\[\sup_{\alpha\in\mathcal{A}}\inf_{\beta\in\mathcal{B}}\left\{-\tr{ \, a_{\alpha\beta}(x) D^2 \varphi(x)}-I_{\alpha\beta}[x,\varphi]+b_{\alpha\beta}(x)\cdot D\varphi(x)+c_{\alpha\beta}(x)u(x)+f_{\alpha\beta}(x)\right\}\leq 0.  \]
A bounded function $u\in {\rm LSC}(\mathbb{R}^d)$ is a viscosity supersolution is similar with the above maximum replaced by minimum and $\leq $ replaced by $\geq$.

A bounded function $u$ is a viscosity solution of \eqref{main} if $u$ is both a viscosity subsolution and viscosity supersolution of \eqref{main}.
\end{definition}

\begin{definition}
Let $g$ be a continuous function in $\Omega^c$. A continuous bounded function $u$ is a viscosity solution of \eqref{main} with boundary data $g$ if it is a viscosity solution of \eqref{main} and $u=g$ in $\Omega^c$ .
\end{definition}

\medskip

Below we define the weak convergence of operators.
\begin{definition}\label{def conv of oper}
Let $U$ be an open subset of $\mathbb{R}^d$. A sequence of operators $\mathcal{I}^m$ is said to converge weakly to $\mathcal{I}$ in $U$ if for any test function $\varphi\in L^\infty(\mathbb R^d)\cap C^2(B_\epsilon(x_0))$ for some $B_\epsilon(x_0)\subset U$, we have
\[\mathcal{I}^m [x,\varphi]\to \mathcal{I}[x,\varphi]\quad\text{ uniformly in }B_{\epsilon/2}(x_0) \text{ as }m\to\infty.\]
\end{definition}

\medskip

At the end, we give the following $C^{\alpha}$ regularity theorem which is a crucial ingredient for proving $C^{1,\alpha}$ regularity.
\begin{lemma}\label{thm:holu1111}
Assume that $-\frac{1}{2}\leq u\leq\frac{1}{2}$ in $\mathbb R^d$ such that $u$ solves
\begin{equation*}
\mathcal{P}^+(D^2u)+\mathcal{P}_{K,r}^+(u)+C_0r|D u|\geq -f^-\quad\text{in $B_1$}
\end{equation*}
and
 \begin{equation*}
 \mathcal{P}^-(D^2u)+\mathcal{P}_{K,r}^-(u)-C_0r|D u|\leq f^+\quad\text{in $B_1$}
 \end{equation*}
 in the viscosity sense for some $C_0\geq 0$ and $f\in L^d(B_1)$. Then there exist universal constants $\epsilon_*$, $\alpha$ and $C$ independent of $r$ such that if $\|f\|_{L^d(B_1)}\leq\epsilon_*$ we have
\begin{equation*}
|u(x)-u(0)|\leq C|x|^\alpha.
\end{equation*}
\end{lemma}
\begin{proof}
If $r=1$, then the theorem is exact Theorem 4.1\cite{mou2018}. The proof follows by almost the same argument as that in Theorem 4.1\cite{mou2018} since the estimate obtained in Corollary 3.14\cite{mou2018} is independent of $r$.
\end{proof}


\section{Weak Convergence Lemma}

In this section we consider $\Omega=B_2$.
Recall \eqref{eqL:Ir}, we study the operators $\mathcal{I}^{r_k}_k:=\mathcal{I}^{r_k}_k(x_k)[x,u]$ with $x_k\in B_1$. Denoting $\tilde{x}_k:=r_k(x-x_k)+x_k$, then
\begin{equation}\label{def Irkk}
    \begin{aligned}
        \mathcal{I}_k^{r_k}\varphi(x)&=\sup_{\alpha\in\mathcal{A}_k}\inf_{\beta\in\mathcal{B}_k}\left\{-\tr{ \, a_{\alpha\beta}(\tilde{x}_k) D^2 \varphi(x)}-I_{\alpha\beta}^{r_k}(x_k)[x,\varphi]+r_kb_{\alpha\beta}(\tilde{x}_k)\cdot D\varphi(x)+\right.\\
        &\quad\quad\left.r_k^2c_{\alpha\beta}(\tilde{x}_k)\varphi(x)+r_k^2f_{\alpha\beta}(\tilde{x}_k)\right\}.
    \end{aligned}
\end{equation}


    
    
\begin{lemma}\label{lem:weak con}
{Suppose we have a sequence of operators $\mathcal{I}_k^{r_k}$ satisfying the requirements (A)(B)(C). Suppose for some $z\in\overline{B_1}$, $B_1\ni x_k\to z$ and $r_k\to 0$ as $k\to\infty$.}
Then $\mathcal{I}_k^{r_k}$ converges weakly to some $\mathcal{I}^0$ in $B_2$ where
\begin{equation}
    \label{lim operator}\mathcal{I}^0u(x):=\sup_{\alpha\in\mathcal{A}}\inf_{\beta\in\mathcal{B}}\{-\tr{ \, a_{\alpha\beta}(z) D^2 u(x)}\}.
\end{equation}
\end{lemma}

\begin{proof}

Fix $x_0\in B_2$ and a test function $\varphi\in C^2(B_\epsilon(x_0))\cap L^\infty(\mathbb R^d)$ for some $B_\epsilon(x_0)\subset B_2$.

It is obvious that $r_kb_{\alpha\beta}(\tilde{x}_k)\cdot D\varphi(x),r_k^2c_{\alpha\beta}(\tilde{x}_k)\varphi(x),r_k^2f_{\alpha\beta}(\tilde{x}_k)\to 0$ uniformly in $x\in B_{\epsilon/2}(x_0),\alpha\in \cup_k\mathcal{A},\beta\in\cup_k\mathcal{B}$. The nonlocal term 
\begin{align*}
    \left|I_{\alpha\beta}^{r_k}[x,\varphi]\right|&\leq\int_{\mathbb{R}^d} \left|\varphi(x+z)-\varphi(x)-\chi_{B_\frac{1}{r_k}}(z)D\varphi(x)\cdot z\right| K^{r_k}(z)dz.\\
    &=\int_{B_{\epsilon/4}}+\int_{B^c_{\epsilon/4}\cap B_{{r_k^{-1/2}}}}+\int_{B^c_{r_k^{-1/2}}}=:I_1+I_2+I_3.
\end{align*}
Then the first part of the integration
\begin{align*}
    I_1
    =&\int_{B_{\epsilon/4}} \left|\varphi(x+z)-\varphi(x)-D\varphi(x)\cdot z\right| K^{r_k}(z)dz\\
    \leq &\int_{B_{\epsilon/4}} \|D^2 \varphi\|_{L^{\infty}(B_{3\epsilon/4}(x_0))}|z|^2 K^{r_k}(z)dz\\
    \leq & \|D^2 \varphi\|_{L^\infty (B_{3\epsilon/4}(x_0))}\int_{B_{r_k\epsilon/4}}|z|^2K(z)dz.
\end{align*}
The second part
\begin{align*}
    I_2
    \leq &C(\|\varphi\|_{L^\infty(\mathbb R^d)}+\|D\varphi\|_{L^\infty(B_{3\epsilon/4}(x_0))})\int_{B^c_{\epsilon/4}\cap B_{{r_k^{-1/2}}}} (1+ | z|) K^{r_k}(z)dz\\
     \leq &C(\|\varphi\|_{L^\infty(\mathbb R^d)}+\|D\varphi\|_{L^\infty(B_{3\epsilon/4}(x_0))})\int_{B^c_{r_k \epsilon/4}\cap B_{{r_k^{1/2}}}} (r_k^2+r_k |z|) K(z)dz\\
     \leq &C(\|\varphi\|_{L^\infty(\mathbb R^d)}+\|D\varphi\|_{L^\infty(B_{3\epsilon/4}(x_0))})(\epsilon^{-2}+\epsilon^{-1})\int_{B^c_{r_k \epsilon/4}\cap B_{{r_k^{1/2}}}} |z|^2K(z)dz.
\end{align*}
The last part
\begin{align*}
    I_3
    \leq &C(\|\varphi\|_{L^\infty(\mathbb R^d)}+\|D\varphi\|_{L^\infty(B_{3\epsilon/4}(x_0))})\int_{B^c_{r_k^{-1/2}}} (1+ \chi_{B_\frac{1}{r_k}}(z)| z|) K^{r_k}(z)dz\\
    \leq &C(\|\varphi\|_{L^\infty(\mathbb R^d)}+\|D\varphi\|_{L^\infty(B_{3\epsilon/4}(x_0))})\int_{B^c_{r_k^{1/2}}}  (r_k^{2}+\chi_{B_1}(z)r_k|z|)K(z)dz\\
     \leq &C(\|\varphi\|_{L^\infty(\mathbb R^d)}+\|D\varphi\|_{L^\infty(B_{3\epsilon/4}(x_0))})r_k^\frac{1}{2}\int_{ B^c_{r_k^{1/2}}}  \min\{|z|^2,1\}K(z)dz.
\end{align*}
Therefore, the integrals $I_1$, $I_2$ and $I_3$ all go to $0$ uniformly in $x\in B_{\epsilon/2}(x_0)$ as $k\to +\infty$.

We conclude that 
\begin{equation}\label{tozero}
   |I_{\alpha\beta}^{r_k}[y,\varphi]|+|r_kb_{\alpha\beta}(y)\cdot D\varphi(y)|+
        |r_k^2c_{\alpha\beta}(y)\varphi(y)|+|r_k^2f_{\alpha\beta}(y)|\leq \omega(r_k,\varphi,\epsilon,K)
\end{equation}
uniformly in $B_{\epsilon/2}(x_0)$
where $\omega(r_k,\varphi,\epsilon,K)\to 0$ as $r_k\to 0$ independent of $\alpha,\beta,x\in B_{\epsilon/2}(x_0)$. 

For the remaining second order term, since \[a_{\alpha\beta}(\tilde{x}_k)\to a_{\alpha\beta}(z) \text{ uniformly in  $x$  as  } r_k\to 0,  \]
we find
\[\mathcal{I}_k^{r_k}\varphi(x)\to \mathcal{I}^0\varphi(x) \text{ uniformly in } x\in B_{\epsilon/2}.\]

\end{proof}


\begin{lemma}\label{lem stability}Let $\{\mathcal{I}^{r_k}_k\}_k\cup\{\mathcal{I}^0\}$ are given by \eqref{def Irkk} and \eqref{lim operator} satisfying (A)(B)(C), and $\{u^k\}_k\cup\{u\}\subset {\rm LSC}(\mathbb R^d)$ be a sequence of uniformly bounded functions in $\mathbb{R}^d$ such that as $k\to\infty$\begin{align*}&r_k\to 0, \quad u^k\to u \text{ in the $\Gamma$ sense in } B_1,\quad\\ &\mathcal{I}^{r_k}_k\to \mathcal{I}^0 \text{ weakly in } B_1,\\
&\mathcal{I}^{r_k}_ku^k(x)\leq \eta_k \text{ in }B_1,\quad \text{ with }\eta_k\to 0 \text{ as }k\to\infty.\end{align*}Then $\mathcal{I}^0u(x)\leq 0$ in $B_1$.

The similar result holds if we replace ${\rm LSC}(\mathbb R^d)$ by ${\rm USC}(\mathbb R^d)$ and the two $``\leq "$s by $``\geq "$s.
\end{lemma}
\begin{proof}

Let $\varphi\in C^2(B_1)\cap C(\mathbb R^d)$ be a bounded text function such that $\varphi$ touches $u$ strictly from below at $x\in B_1$. Since $u^k$ $\Gamma$-converges to $u$ in $B_1$, we can find $x_k$ and $d_k$ such that $\varphi+d_k$ touches $u_k$ from below locally at $x_k$ where $x_k\to x, d_k\to 0$ as $k\to\infty $. Let $\varphi_k=\varphi+d_k$ and thus
\[\mathcal{I}^{r_k}_k\varphi^k(x_k)\leq \eta_k \text{ for $k$ large enough}.\] The goal is to show $\mathcal{I}^0\varphi(x)\leq 0$.

We have
\begin{align*}
    \mathcal{I}^0\varphi(x)&\leq |\mathcal{I}^0\varphi(x)-\mathcal{I}^{r_k}_k\varphi(x)|+|\mathcal{I}^{r_k}_k\varphi(x)-\mathcal{I}^{r_k}_k\varphi^k(x_k)|+\mathcal{I}^{r_k}_k\varphi^k(x_k)\\
    &=:I_1+I_2+I_3.
\end{align*}
By the weak convergence assumption of the operator $\mathcal{I}_k^{r_k}$, we have $I_1\to 0$ as $k\to \infty$. And $I_3\leq \eta_k\to 0$, as $k\to\infty$, is given by the assumption.

Now we show $I_2\to 0$. Recall the definition of $\mathcal{I}_k^{r_k}$ in \eqref{def Irkk}. It is not hard to see that
\[|\mathcal{I}^{r_k}_k\varphi^k(x_k)-\mathcal{I}^{r_k}_k\varphi(x_k)|\leq C|d_k|\]
which goes to $0$ as $k\to \infty$.
We only need to show
$|\mathcal{I}^{r_k}_k\varphi(x)-\mathcal{I}^{r_k}_k\varphi(x_k)|\to 0.$ This is guaranteed by the requirements (A)(C), $\varphi\in C^2(B_1)\cap C(\mathbb R^d)$ and $x_k\to x$.
We then conclude with $I_2\to 0$.

\end{proof}

\section{$C^{1,\alpha}$ Regularity}


The main theorem of this section is a $C^{1,\alpha}$ regularity for solutions of \eqref{main}. Before proving it, we first give the following Lemma.
\begin{lemma}\label{compare}
For any $x_0\in B_1$ and $r>0$, let $\mathcal{I}^r(x_0)$ and $\mathcal{I}^0(x_0)$ are given in \eqref{eqL:Ir} where the corresponding coefficients satisfy the requirements (A)(B)(C) with $\Omega=B_2$. Suppose that given $M,\epsilon>0$ and a modulus of continuity $\rho$, assume that there exist $r_0, \eta>0$ indepent of $x_0$ such that
\begin{align*}
    &r\leq r_0,\,\mathcal{I}^0(x_0)[x,v]=0 \text{ in }B_1,\\
    &\mathcal{I}^{r}(x_0)[x,u]\geq -\eta,\, \mathcal{I}^{r}(x_0)[x,u]\leq \eta \text{ in }B_1,\\
    &u=v \text{ in $\partial B_1$},\, |u(x)|+|v(x)|\leq M,\\
    &|u(x)-u(y)|+|v(x)-v(y)|\leq \rho(|x-y|) \text{ for all }x,y\in \overline{B_1}.
\end{align*}

Then
\[|u-v|\leq \epsilon \text{ in }B_1.\]
\end{lemma}

\begin{proof}
The proof is very much the same as Lemma 7 \cite{LL2}. 
If the statement is false, there exists \[M,\epsilon, x_k, r_k,\eta_k,u_k,v_k, x_0\] such that
\begin{align*}
&  x_k\in \overline{B_1}\,\,,\,\, x_0\in \overline {B_1}\,\,\text{and}\,\,x_k\to x_0,\\
    &  r_k\to 0\,\,\text{and}\,\, \eta_k\to 0,\\
    &\text{the assumptions of the lemma are valid for each }k,\\
    &\sup |u_k-v_k|\geq\epsilon \text{ in }B_1.
\end{align*}
By Lemma \ref{lem:weak con}, $\mathcal{I}^{r_k}(x_k)$ converges weakly to $\mathcal{I}^0(x_0)$ in $B_1$. 

Since $u_k,v_k$ share the same modulus of continuity in $\overline B_1$, by passing to a subsequence, we may assume that 
\[u_k\to u,\quad v_k\to v \text{ in } L^\infty(B_1)\]
where $u,v$ are bounded continuous function in $\overline{B_1}$. Also by the assumption that $u_k=v_k$ on $\partial B_1$, we have 
\[u=v=:g \text{ on }\partial B_1.\]

By Lemma \ref{lem stability}, $u,v$ solves 
\begin{equation*}
    \left\{\begin{aligned}
        & \mathcal{I}^0(x_0) [x,w]=0 \quad &\text{ in }&B_1,\\
        & w=g \quad &\text{ on }&\partial B_1.
    \end{aligned}\right.
\end{equation*}
By comparison principle, we derive that $u=v$ which contradicts 
\[\sup |u_k-v_k|\geq\epsilon \text{ in }B_1.\]
Then the proof follows.

\end{proof}

The following theorem shows the $C^{1,\alpha}$ estimates for \eqref{main}. The proof follows from the idea of Theorem 52\cite{LL2}. Comparing to the $C^{1,\alpha}$ regularity result in \cite{LL2}, the major differences are the following: 1. since our operator $\mathcal{I}$ is non-scale invariant, we have to use a uniform interior $C^\alpha$ estimate for all scaling operators $\{\mathcal{I}^{r}\}_r$; 2. with the weak integrability assumption \eqref{integ K} on K, the nonlocal operator $I_{\alpha\beta}$ is not well defined acting on any unbounded function, and thus we need to cut off our solutions and take care of the errors generated by the cut-off. 


\begin{theorem}\label{thm 3.1}
Let $\mathcal{I}$ given in \eqref{main} satisfy (A)(B)(C) with $\Omega=B_2$. Suppose $u\in L^\infty(\mathbb R^d)$ solves $\mathcal{I}[x,u]=0$ in $B_2$ in the viscosity sense. Then there exist constants $0<\alpha<1$ and $C>0$ such that  
\[\|u \|_{C^{1,\alpha}(B_{1})}\leq C\left(\|u\|_{L^\infty(\mathbb R^d)}+\sup_{\alpha\in\mathcal{A},\beta\in\mathcal{B}}\|f_{\alpha\beta}\|_{L^\infty(B_2)}\right),\]
where $C$ depends on $\lambda$, $\Lambda$, $\sup_{\alpha\in\mathcal{A},\beta\in\mathcal{B}}\|b_{\alpha\beta}\|_{L^\infty(B_2)}$, $\sup_{\alpha\in\mathcal{A},\beta\in\mathcal{B}}\|c_{\alpha\beta}\|_{L^\infty(B_2)}$, $K$ and $d$.
\end{theorem}
\begin{proof}
Fix any $x_0\in B_{1}$, we remind you that
\[ \mathcal{I}^r(x_0) [x,u]=r^2\mathcal{I}[x_0+r(x-x_0),u(x_0+r^{-1}(\cdot-x_0))]. \]
By Lemma \ref{lem:weak con}, as $r_k\to 0$ 
\[\mathcal{I}^{r_k}(x_0)\to \mathcal{I}^0(x_0).\]
In particular $\mathcal{I}^0(x_0)$ has interior $C^{1,\beta}$ estimates for some universal constant $\beta>0$.

The proofs and constants below will be independent of $x_0$ since our Lemma \ref{compare} is independent of $x_0$. For simplicity let us assume $x_0=0 $. Also without loss of generality, we can assume that
\begin{equation}\label{est 3.0}
\|u\|_{L^\infty(\mathbb{R}^d)}\leq 1.
\end{equation}
Using Lemma \ref{thm:holu1111}, we have $u\in C^\beta(B_1)$.

As done in Theorem 52 \cite{LL2}
we will show that there is a $\delta,\nu\in (0,\frac{1}{4})$ and a sequence of linear functions $l_k(x)=a_k+b_kx$ such that
\begin{align}
    &\sup_{B_{2\delta\nu^k}}|u-l_k|\leq \nu^{k(1+\alpha)},\label{cond1}\\
    &|a_{k}-a_{k-1}|\leq \nu^{(k-1)(1+\alpha)},\label{cond2}\\
    &\nu^{k-1}|b_{k}-b_{k-1}|\leq C\nu^{(k-1)(1+\alpha)}\label{cond3}\\
    \text{ and }&|u-l_k|\leq \nu^{-k(\alpha'-\alpha)}\delta^{-1-\alpha'}|x|^{1+\alpha'} \text{ for }x\in B^c_{2\delta\nu^k}.\label{cond4}
\end{align}
Here we took any $\alpha,\alpha'$ such that $0<\alpha<\alpha'<\beta$.
Once this is done, it is standard to obtain the $C^{1,\alpha}$ estimate of $u$ (at the origin) follows. 

When $k=0$, let $l_{-1}=l_0=0$. Since \eqref{est 3.0} holds, we have \eqref{cond1}-\eqref{cond4}. We proceed by induction. Assume \eqref{cond1}-\eqref{cond4} are satisfied up to some $k$ and we will show \eqref{cond1}-\eqref{cond4} for $k+1$.

Let $\xi:\mathbb{R}^d\to [0,1]$ such that $\xi$ is continuous and \begin{align*}
    \xi(x)=1 \text{ for }|x|\leq 3,\quad \xi(x)=0 \text{ for }|x|\geq 4.
\end{align*}
We define
\[w_k(x):=\frac{(u-\xi l_k)(\delta\nu^k x)}{\nu^{k(1+\alpha)}} \text{ and } w_k'(x):=\max\{\min\{w_k(x),1\},-1\}.\]
The reason we define $w_k'$ is that $w_k'$ is uniformly bounded independent of $k$.

By \eqref{cond1}, $|w_k'|=|w_k|\leq 1$ in $B_2$.
 Then in $B_2$
\begin{align*}
    \mathcal{I}_k w_k'(x)&:=\mathcal{I}^{\delta\nu^k}(0)[x,w_k']=\sup_\alpha\inf_\beta\{-\tr{ \, a_{\alpha\beta}(x) D^2 w_k(x)}-I^{\delta\nu^k}_{\alpha\beta}(0)[x,w_k']+ \\
    &\quad\quad \delta\nu^{k}b_{\alpha\beta}(x)\cdot Dw_k(x)+ \delta^2\nu^{2k}c_{\alpha\beta}(x)w_k(x)
+\delta^2\nu^{2k}f_{\alpha\beta}(x)\}.
\end{align*}
By the inductive requirements we have on $a_k,b_k$, they are uniformly bounded.
Since $\|u\|_\infty\leq 1$ and $ \xi l_k$ is uniformly bounded, $|w_k|\leq C\nu^{-k(1+\alpha)}$ in $\mathbb{R}^d$.
By \eqref{cond4}, for all $x\in B_2^c\cap B_{2\delta^{-1}\nu^{-k}}$, 
\begin{equation*}
|w_k(x)|\leq |x|^{1+\alpha'}.
\end{equation*}
And, for any $x\in B_{2\delta^{-1}\nu^{-k}}^c$,
\begin{equation*}
|w_k(x)|\leq C\nu^{-k(1+\alpha)}\leq C\delta^{-(1+\alpha')}\nu^{-(1+\alpha')k}\leq C|x|^{1+\alpha'}.
\end{equation*}
Also since $w'_k$ is bounded, we can assume that for $x\in B_2^c$
\begin{equation}
    \label{w' vs w}
    |w_k|+|w_k'-w_k|\leq C\min\{|x|^{1+\alpha'},\nu^{-k(1+\alpha)}\}.
\end{equation}
Thus, if we restrict $x$ in a smaller ball $B_{3/2}$ and denote by $s:=\delta\nu^k(\leq \delta)$
\begin{align*}
    &\quad\quad |I^{\delta\nu^k}_{\alpha\beta}(0)[x,w'_k-w_k]|\\
    &\leq \int_{\{z: |x+z|\geq 2\}\cap B_{s^{-1}}}|w'_k-w_k|(x+z) K^{s}(z)dz+\int_{B_{s^{-1}}^c}|w'_k-w_k|(x+z) K^{s}(z)dz\\
    &\leq C\left(\int_{B_{1/2}^c\cap B_{s^{-1}}}|x+z|^{1+\alpha'} K^{s}(z)dz+\int_{B_{s^{-1}}^c}\nu^{-k(1+\alpha)} K^{s}(z)dz \right)\\
    &\leq C\left(\int_{B_{1/2}^c\cap B_{s^{-1}}}|z|^{1+\alpha'} K^{s}(z)dz+\delta^2\nu^{k(1-\alpha)}\int_{B_{s^{-1}}^c}\delta^{-2}\nu^{-2k} K^{s}(z)dz\right)\\
    &\leq C\left(\int_{B_{1/2}^c\cap B_{s^{-1/2}}}|z|^{2} K^{s}(z)dz+s^{(1-\alpha')/2}\int_{B_{s^{-1/2}}^c\cap B_{s^{-1}}}|z|^{2} K^{s}(z)dz+\delta^2\nu^{k(1-\alpha)}\int_{B_{s^{-1}}^c}\delta^{-2}\nu^{-2k} K^{s}(z)dz\right)\\
    &\leq  C\left(\int_{B_{s/2}^c\cap B_{s^{1/2}}}|z|^{2} K(z)dz+s^{(1-\alpha')/2}+\delta^2\nu^{k(1-\alpha)}\right)\\
    &\leq C\left( \int_{B_{s^{1/2}}}|z|^{2} K(z)dz+s^{(1-\alpha')/2}+\delta^2\nu^{k(1-\alpha)}\right)\\
    &\leq C\left(\int_{B_{\delta^{1/2}}}|z|^{2} K(z)dz+\delta^{(1-\alpha')/2}+\delta^2\right)\\
    &=: \omega_1(\delta)
\end{align*}
where $\omega_1(\delta)\to 0$ as $\delta\to 0$.
Then for any $x \in B_{3/2}$, we have
\begin{align*}
     \mathcal{I}_k w_k(x)&=\delta^2\nu^{k(1-\alpha)} \sup_{\alpha\in\mathcal{A}}\inf_{\beta\in\beta}\{-\tr{ \, a_{\alpha\beta}D^2 u}-I_{\alpha\beta}[\cdot,u-\xi l_k]+b_{\alpha\beta}\cdot Du\\
     &\quad +c_{\alpha\beta}u
+f_{\alpha\beta}-b_{\alpha\beta}\cdot b_k -c_{\alpha\beta} l_k\}(\delta\nu^k x)\\
&=\delta^2\nu^{k(1-\alpha)} \sup_{\alpha\in\mathcal{A}}\inf_{\beta\in\mathcal{B}}\{-\tr{ \, a_{\alpha\beta}D^2 u}-I_{\alpha\beta}[\cdot,u]+b_{\alpha\beta}\cdot Du\\
     &\quad +c_{\alpha\beta}u
+f_{\alpha\beta}-b_{\alpha\beta}\cdot b_k -c_{\alpha\beta} l_k+I_{\alpha\beta}[\cdot,\xi l_k]\}(\delta\nu^k x).
\end{align*}

Now for any $0<\epsilon<1$ to be determined, let $r_0:=r_0(\epsilon),\eta:=\eta(\epsilon)$ be given in Lemma \ref{compare}.
Since $\mathcal{I}[x,u]=0$ in $B_2$ and
\[b_{\alpha\beta}\cdot b_k,c_{\alpha\beta}l_k, |{I}_{\alpha\beta}[\cdot,\xi l_k]|\]
are uniformly bounded in $B_{2\delta\nu^k}$, we have in $B_{3/2}$
\[|\mathcal{I}_kw_k(x)|\leq C\delta^2\nu^{k(1-\alpha)}\leq C\delta^2.\] Therefore in $B_{3/2}$
\begin{equation}\label{Ik w'}
    |\mathcal{I}_kw_k'(x)|\leq C\delta^2+\omega_1(\delta)<\eta 
\end{equation}
if $\delta:=\delta(\epsilon)$ is taken to be small. Also we can require
$\delta\leq r_0$.

Because $w_k'$ satisfies \eqref{Ik w'}, again using Lemma \ref{thm:holu1111}, we have $\|w_k'\|_{C^\beta(\overline{B_1})}\leq C$ for some $C$ independent of $k$. 
Then let us consider a function $h$ which solves
\begin{equation*}
    \left\{
    \begin{aligned}
        &\mathcal{I}^0(x_0) [x,h]=0 &\text{ in }B_1,\\
        & h=w_k' &\text{ on }\partial B_1.
    \end{aligned}
    \right.
\end{equation*}
By Lemma \ref{compare}, $|w_k'-h|\leq \epsilon$ in $B_1$. By $C^{1,\beta}$ estimates, we can take $\bar{l}=\bar{a}+\bar{b}x$ to be the linear part of $h$ at the origin. Since $|w_k'|\leq 1$ in $B_1$ and $0<\epsilon<1$, we have $|\bar{a}|\leq 1+\epsilon$. By the $C^{1,\beta}$ regularity for $\mathcal{I}^0(x_0)$, we also have $|\bar{b}|\leq C$ for some $C$ independent of $k$. 

Also we have for some constant $C_1>0$ independent of $k$
\[|h(x)-\bar{l}(x)|\leq C_1 |x|^{1+\beta}\quad \text{ in }B_{1/2}. \]
So
\begin{equation}\label{est 1/2}
    |w_k(x)-\bar{l}(x)|\leq \epsilon+C_1|x|^{1+\beta}\quad \text{ in }B_{1/2}.
\end{equation}
Using \eqref{w' vs w}, we have
\begin{align}
    &|w_k(x)-\bar{l}(x)|\leq 1+\bar a+\bar b\leq C &\text{in }&B_1,\label{est 1}\\
    &|w_k(x)-\xi(\delta\nu^k x)\bar{l}(x)|\leq |w_k(x)|+|\bar{l}(x)|\leq C|x|^{1+\alpha'}+C|x| &\text{ in }& B^c_1.\label{est inf}
\end{align}

We define

\begin{align*}
&l_{k+1}(x):=l_k(x)+\nu^{k(1+\alpha)}\bar{l}(\frac{x}{\delta\nu^k}),\\
&w_{k+1}(x):=
\frac{(u-\xi\, l_{k+1})(\delta\nu^{k+1} x)}{\nu^{(k+1)(1+\alpha)}}\\
 &(\text{ where }w_{k+1}(x)=\frac{(w_k-\bar{l})(\nu x)}{\nu^{1+\alpha}} \text{ if }x\leq \delta^{-1}\nu^{-k}).
\end{align*}
Then \eqref{cond2}\eqref{cond3} for $k+1$ follows.

Now let $\nu$ be sufficiently small such that $\max\{\nu^{\beta-\alpha}, C_1\nu^{\alpha'-\alpha}\}\leq 1/8$ and then we take $\epsilon\leq \nu^{1+\beta}$.
Using \eqref{est 1/2}\eqref{est 1}\eqref{est inf}, we have 
\begin{align}
    &|w_{k+1}(x)|\leq \nu^{\beta-\alpha}+C_1\nu^{\alpha'-\alpha}|x|^{1+\alpha'}\leq  \frac{1}{8}(1+|x|^{1+\alpha'}) \text{ in }B_{\nu^{-1}/2},\label{est con1}\\
    &|w_{k+1}(x)|\leq C\nu^{-(1+\alpha)} \text{ for }x\in B_{\nu^{-1}}\backslash B_{\nu^{-1}/2},\nonumber\\
    &|w_{k+1}(x)|\leq \frac{|w_k(\nu x)-\xi(\delta \nu^{k+1}x) \bar{l}(\nu x) |}{\nu^{(1+\alpha)}}\leq C|\nu|^{\alpha'-\alpha}|x|^{1+\alpha'}+C\nu^{-\alpha}|x| \text{ for }x\in\mathbb{R}^d\backslash B_{\nu^{-1}}.\nonumber
\end{align}
From the above estimates, we can further choose $\nu$ small enough such that 
\begin{equation}
\label{est B out}
    |w_{k+1}(x)|\leq \frac{1}{2}|x|^{1+\alpha'} \text{ for all }x\in B_2^c.
\end{equation}

Since 
\[w_{k+1}(x)=\frac{(u-\xi l_{k+1})(\delta\nu^{k+1} x)}{\nu^{(k+1)(1+\alpha)}}=\frac{(u- l_{k+1})(\delta\nu^{k+1} x)}{\nu^{(k+1)(1+\alpha)}}\]
in $B_2$, by \eqref{est con1} and $\nu<\frac{1}{4}$, {$|w_{k+1}(x)|\leq 1 \text{ in } B_2$} and thus \eqref{cond1} holds for $k+1$.

For $x\in B_{2\delta\nu^{k+1}}^c$, by \eqref{est B out} we have
\begin{align*}
    |(u-l_{k+1})(x)|&\leq \frac{1}{2} \nu^{(k+1)(1+\alpha)}|\delta^{-1}\nu^{-k-1}x|^{1+\alpha'}\\
    &\leq \nu^{-(k+1)(\alpha'-\alpha)}\delta^{-1-\alpha'}|x|^{1+\alpha'}.
\end{align*}
So \eqref{cond4} holds for $k+1$. This completes the inductive step and the proof.
\end{proof}

\begin{corollary}\label{col:C1alpha}
Let $\mathcal{I}$ given in \eqref{main} satisfy (A)(B)(C). Suppose $u\in L^\infty(\mathbb R^d)$ solves $\mathcal{I}[x,u]=0$ in $\Omega$ in the viscosity sense. Then there exist constants $0<\alpha<1$ and $C>0$ such that  
\[\|u \|_{C^{1,\alpha}(\Omega_\delta)}\leq C\left(\|u\|_{L^\infty(\mathbb R^d)}+\sup_{\alpha\in\mathcal{A},\beta\in\mathcal{B}}\|f_{\alpha\beta}\|_{L^\infty(\Omega)}\right),\]
where $C$ depends on $\delta$, $\lambda$, $\Lambda$, $\sup_{\alpha\in\mathcal{A},\beta\in\mathcal{B}}\|b_{\alpha\beta}\|_{L^\infty(\Omega)}$, $\sup_{\alpha\in\mathcal{A},\beta\in\mathcal{B}}\|c_{\alpha\beta}\|_{L^\infty(\Omega)}$, $K$ and $d$.

\end{corollary}

\section{$C^{2,\alpha}$ Regularity}

In this section we will show that the following Dirichlet problem
\begin{equation}\label{convex operator}
     \left\{\begin{aligned}
  &\tilde{\mathcal{I}}[x,u]:=\sup_{\alpha\in\mathcal{A}}\{-\tr{ \, a_{\alpha}(x) D^2 u(x)}- \tilde{I}_{\alpha}[x,u]+b_{\alpha}(x)\cdot Du(x)+c_{\alpha}(x)u(x)+f_{\alpha}(x)\}=0,\,\,\text{in $\Omega$},\\
&u=g,\,\, \text{ in }\Omega^{c},
     \end{aligned}\right.
\end{equation}
with a bounded function $g\in C(\mathbb R^d)$ and
\begin{equation}
    \tilde{I}_{\alpha}[x,u]=\int_{\mathbb{R}^d}( u(x+z)-u(x)-\chi_{B_1}(z)Du(x)\cdot z)N_{\alpha}(x,z)dz,
\end{equation}
admits a unique viscosity solution $u$, which is $C_{\rm loc}^{2,\alpha}(\Omega)$.

To study \eqref{convex operator}, we first consider
\begin{equation}\label{modified operator}
    \left\{\begin{aligned}&
\tilde{\mathcal{I}}_q [x,u_q]=0 &\text{ in }&\Omega,\\
&u=g &\text{ in }&\Omega^{c},
\end{aligned}\right.
\end{equation}
where $\tilde{\mathcal{I}}_q$ is defined by replacing $\tilde{I}_\alpha$ by the truncated nonlocal operator
 \[ \tilde{I}_{\alpha,q}[x,u]:=\int_{\mathbb{R}^d}( u(x+z)-u(x)-\chi_{B_1}(z)Du(x)\cdot z)N_{\alpha}(x,z)\chi_{|z|\geq 1/q}\,dz\]
 in \eqref{convex operator}. The idea is to use the $C^{1,\alpha}$ estimates obtained in the previous section to derive the uniform $C^{2,\alpha}$ estimates for $u_q$, which allows us to pass the limit up a subsequence and obtain a classical solution $u$ of \eqref{convex operator}. Comparing viscosity and classical solutions, we know that $u$ is the unique viscosity solution of \eqref{convex operator}. We first start with the following a priori estimates.
 
 \begin{theorem}\label{a priori estimate}
 Suppose (A')(B)(C') hold and $\alpha$ is a sufficiently small universal constant. If $u\in C^2(\Omega)\cap C^{\alpha}(\mathbb R^d)$ solves $\tilde{\mathcal{I}}[x,u]=0$ in $\Omega$ classically, then
 \begin{equation*}
 \|u\|_{C^{2,\alpha}(\Omega_\delta))}\leq C\left(\|u\|_{C^{\alpha}(\mathbb R^d)}+\sup_{\alpha\in\mathcal{A}}\|f_{\alpha}\|_{C^\alpha(\Omega)}\right)
 \end{equation*}
 where $C$ depends on $\delta$, $\lambda$, $\Lambda$, $\sup_{\alpha\in\mathcal{A}}\|b_\alpha\|_{C^\alpha(\Omega)}$, $\sup_{\alpha\in\mathcal{A}}\|c_\alpha\|_{C^\alpha(\Omega)}$, $K$ and $d$.
 \end{theorem}
 \begin{proof}
 Let $\psi\in C_c^\infty(\Omega)$ and $0\leq\psi\leq 1$. For some $1>\tilde\delta>0$, assume 
\[ supp\{\psi\}\subset\Omega_{2\tilde\delta}, \, \Omega_{3\tilde\delta}\subset\{\psi=1\} .\]
Then $\psi u$ is a classical solution of
\begin{equation}\label{eqn elli uq}
    \sup_\alpha\{-\tr{ \, a_{\alpha}(x) D^2 (\psi u)(x)}+b_{\alpha}(x)\cdot D(\psi u)(x)+c_{\alpha}(x)(\psi u)(x)+\tilde{f}_{\alpha}(x)\}=0
\end{equation}
where
\begin{align*}
    \tilde{f}_\alpha(x):=&  u(x)\tr\, a_\alpha(x) D^2 \psi(x)+\tr\, a_\alpha(x) D\psi(x)\otimes Du(x)-u(x) b_\alpha(x)\cdot D\psi(x)+f_\alpha(x) \psi(x)-\tilde{I}_{\alpha}[x,\psi u]\\
    &+\int_{\mathbb{R}^d}( \psi(x+z)-\psi(x)-\chi_{B_1}(z)D\psi(x)\cdot z)u(x+z)N_{\alpha}(x,z)\,dz\\
    &+\int_{\mathbb{R}^d}( u(x+z)-u(x))D\psi(x)\cdot z N_{\alpha}(x,z)\chi_{B_1}(z)\,dz.
\end{align*}
Applying Corollary \ref{col:C1alpha}, we obtain $\|u\|_{C^{1,\alpha}(\Omega_{\tilde\delta})}\leq C\left(\|u\|_{L^\infty(\mathbb R^d)}+\sup_{\alpha\in\mathcal{A}}\|f_\alpha\|_{L^\infty(\Omega)}\right)$.
Then, using $u\in C_{\rm loc}^{1,\alpha}(\Omega)$, (A'), (C') and $\psi\in C_c^{\infty}(\Omega)$, we know
\[b_{\alpha}(x)\cdot D(\psi u)(x),\,c_{\alpha}(x)(\psi u)(x),\,u(x)\tr\, a_\alpha(x) D^2 \psi(x),\,\]\[\tr\, a_\alpha(x) D\psi(x)\otimes Du(x),\,u(x) b_\alpha(x)\cdot D\psi(x),\,f_\alpha(x) \psi(x)\]
are uniformly $C^\alpha$ in $\Omega$. Then, by $\psi\in C_c^\infty(\Omega)$, $u\in C^{\alpha}(\mathbb R^d)$ and (C'), we have
\begin{align*}
&\int_{\mathbb{R}^d}( \psi(x+z)-\psi(x)-\chi_{B_1}(z)D\psi(x)\cdot z)u(x+z)N_{\alpha}(x,z)\,dz\\
=&\int_{B_1}2\Sigma_{|\gamma|=2}\frac{z^\gamma}{\gamma !}\left(\int_0^1 (1-t)\partial^\gamma \psi(x+tz)dt\right)u(x+z)N_{\alpha}(x,z)\,dz\\
&+\int_{B_1^c}( \psi(x+z)-\psi(x))u(x+z)N_{\alpha}(x,z)\,dz
\end{align*}
is uniformly $C^\alpha$ in $\Omega$.
Also 
\begin{align*}
   & \int_{\mathbb{R}^d}( u(x+z)-u(x))D\psi(x)\cdot z N_{\alpha}(x,z)\chi_{B_1}(z)\,dz\\
   =&\int_{B_{\tilde\delta}}\Sigma_{|\gamma|=1}z^\gamma\left(\int_0^1 (1-t)\partial^\gamma u(x+tz) dt\right)\, D\psi(x)\cdot zN_\alpha(x,z)dz\\
   &+\int_{B_1\setminus B_{\tilde\delta}^c}( u(x+z)-u(x))D\psi(x)\cdot z N_{\alpha}(x,z)dz
\end{align*}
is uniformly $C^\alpha$ in $\Omega$ since $u\in C_{\rm loc}^{1,\alpha}(\Omega)\cap C^\alpha(\mathbb R^d), \psi\in C_c^{\infty}(\Omega)$ and (C') holds.

As for $\tilde{I}_\alpha[x,\psi u]$, we decompose it into two terms, i.e. for $0<\epsilon<\tilde\delta$
\begin{align*}
    \tilde{I}_{\alpha,q}[x,\psi u] =& \int_{B_\epsilon}( (\psi u)(x+z)-(\psi u)(x)-D(\psi u)(x)\cdot z)N_{\alpha}(x,z)\,dz\\
    &+ \int_{B_\epsilon^c}( (\psi u)(x+z)-(\psi u)(x)-\chi_{B_1}(z)D(\psi u)(x)\cdot z)N_{\alpha}(x,z)\,dz=:I_1(x)+I_2(x).
\end{align*}
The second part $I_2$ is uniformly $C^\alpha$ in $\Omega$ since $\psi u\in C^{1,\alpha}(\mathbb R^d)$ and (C') holds. The first part
\[I_1(x)=\int_{B_\epsilon}2\Sigma_{|\gamma|=2}\frac{z^\gamma}{\gamma !}\left(\int_0^1 (1-t)\partial^\gamma (\psi u)(x+tz)dt\right)N_\alpha(x,z) \,dz.\]
We find for $x,y\in\Omega_{\tilde\delta}$
\begin{align*}
    |I_1(x)-I_1(y)|&\leq \int_{B_\epsilon}2\Sigma_{|\gamma|=2}\frac{z^\gamma}{\gamma !}\left(\int_0^1 (1-t)\partial^\gamma \left|(\psi u)(x+tz)-(\psi u)(y+tz)\right|dt\right)K(z) \,dz\\
    &\leq C \int_{B_\epsilon}|z|^2\int_0^1 \left\|(\psi u)\right\|_{C^{2,\alpha}(\Omega_{\tilde\delta})}|x-y|^\alpha dtK(z) \,dz\\
    &\leq c_\epsilon \left\|(\psi u)\right\|_{C^{2,\alpha}(\Omega)}|x-y|^\alpha\\
    &= c_\epsilon \left\|(\psi u)\right\|_{C^{2,\alpha}(\Omega_{2\tilde\delta})}|x-y|^\alpha 
\end{align*}
where $c_\epsilon$ is a universal constant depending only on $\epsilon$ and it converges to $0$ as $\epsilon\to 0$.

We proved that 
\[\|\tilde{f}_\alpha\|_{C^{\alpha}(\Omega_{\tilde\delta})}\leq C(\tilde\delta,\epsilon)\left(\|u\|_{C^\alpha(\mathbb R^d)}+\|f\|_{C^\alpha(\Omega)}\right)+c_\epsilon\|(\psi u)\|_{C^{2,\alpha}(\Omega_{2\tilde\delta})}.\]
Using the interior $C^{2,\alpha}$ estimates for local elliptic equations, we have
\[\|\psi u\|_{C^{2,\alpha}(\Omega_{2\tilde\delta})}\leq C(\tilde\delta,\epsilon)\left(\|u\|_{C^\alpha(\mathbb R^d)}+\|f\|_{C^\alpha(\Omega)}\right)+c(\tilde\delta,\epsilon)\|(\psi u)\|_{C^{2,\alpha}(\Omega_{2\tilde\delta})}\]
where $c(\tilde\delta,\epsilon)$ is a constant such that $c(\tilde\delta,\epsilon)\to 0$ as $\epsilon\to 0$ for each $\tilde\delta>0$.
If select $c(\tilde\delta,\epsilon)\leq 1/2$ and $\delta:=3\tilde\delta$, we obtain the desired result. \end{proof}

\begin{lemma}\label{lem:bdry estimate}
Suppose $g\in C^\kappa(\Omega^c)$, (A')(B)(C') hold and $\Omega$ satisfies the uniform exterior ball condition. 
Then, if $u$ solves \eqref{convex operator} in the viscosity sense, we have $u\in C^\alpha(\mathbb R^d)$ for some $\alpha>0$.
\end{lemma}
\begin{proof}
From the proof of Theorem 5.6 \cite{mou2017}, there exist $C_2,\gamma>0$ such that, for any $x\in\partial\Omega$ and $y\in\mathbb{R}^d$, we have
\[|u(x)-u(y)|\leq C_2|x-y|^{\gamma}.\]

The rest of the proof follows from Lemma 3\cite{LL2}.
Fix $x_0\in\Omega$, $y\in\mathbb{R}^d$ and let $2r=d(x_0,\Omega^c)=|x_0-x_1|$ for some $x_1\in\partial\Omega$. If $|x_0-y|\geq r/2$, we have 
\[|y-x_1|\leq |y-x_0|+|x_0-x_1|\leq 5|x_0-y|,\]
and thus
\begin{align*}
    |u(y)-u(x_0)|&\leq |u(y)-u(x_1)|+|u(x_1)-u(x_0)| \\
    &\leq C_2\left( |y-x_1|^{\gamma}+|x_1-x_0|^{\gamma}\right)\\
    &\leq C |y-x_0|^{\gamma}.
\end{align*}

Now consider the case that $|x_0-y|\leq r/2$. Without loss of generality, we assume $0=x_0\in\Omega$ and $\|u\|_{L^\infty(\mathbb R^d)}=1/2$. Define $\rho(x):=C_2|x|^
\gamma$, \begin{equation}
    v(x):=u(rx)-u(x_1) \text{ and }\bar{v}(x):=\min\left\{\rho(4r),\max\left\{v(x),-\rho(4r)\right\}\right\}.
\end{equation} 
Then we have $|x_1|=2r$, $v(x_1/r)=0$ and \begin{equation}\label{est v b2}v(x)=\bar{v}(x) \quad\text{ for all $x\in B_2\subset B_{4}(x_1/r)
$. }\end{equation} 
Notice that, for any $x\in \mathbb{R}^d$, we have \[|v(x)|=|u(rx)-u(x_1)|\leq \rho(|rx-x_1|)\leq\rho(2r+r|x|)\]
and $\|v\|_{L^\infty(\mathbb R^d)}\leq 1$.
And thus we have
\begin{equation}\label{bar v}
    |v(x)-\bar{v}(x)|\leq \min\{(\rho(2r+r|x|)-\rho(4r))_+,\,2\} \text{ in }\mathbb{R}^d.
\end{equation}

Using \eqref{est v b2} and \eqref{bar v}, we have for any $x\in B_1$ 
\begin{align*}
   \left| \tilde{I}^r_\alpha(0)[x, v-\bar{v}]\right|&\leq \int_{|x+z|\geq 2}\min\{(\rho(2r+r|x+z|)-\rho(4r))_+,\,2\}K^r(z)dz\\
    &\leq \int_{|z|\geq 1}\min\{(\rho(3r+r|z|)-\rho(4r))_+,\,2\}K^r(z)dz\\
     &\leq C\int_{|z|\geq 1}\min\{|rz|^{\gamma},\,1\}K^r(z)dz\\
      &= Cr^2\int_{|z|\geq r}\min\{|z|^{\gamma},\,1\}K(z)dz\\
       &\leq Cr^{{\gamma}}\int_{|z|\geq r}\min\{|z|^{2},\,1\}K(z)dz\leq Cr^{\gamma}.
\end{align*}
Thus we proved in $B_1$,
\begin{align}\label{eqL:Ir1}&|\sup_{\alpha\in\mathcal{A}}\left\{-\tr{ \, a_{\alpha}(rx) D^2 \bar{v}(x)}-\tilde{I}^r_{\alpha}(0)[x,\bar{v}]+r\,b_{\alpha}(rx)\cdot D\bar{v}(x)\right.\nonumber\\ &\quad\quad\left.+r^2\,c_{\alpha}(rx)(\bar{v}(x)+u(x_1))+r^2\,f_{\alpha}(rx)\right\}|\leq Cr^{\gamma}. \end{align} 
Using Lemma \ref{thm:holu1111}, we have for $x\in B_{1/2}$
\[|\bar{v}(x)-\bar{v}(0)|\leq C(Cr^{\gamma} +\rho(4r))|x|^{\beta}=Cr^{\gamma}|x|^{\beta}\]
and thus for any $y\in B_{r/2}$
\[|u(y)-u(0)|\leq C r^{\gamma} (\frac{|y|}{r})^{\beta}.\]

If $\beta>{\gamma}$, using $y\in B_{r/2}$ we have  \[ C r^{\gamma} (\frac{|y|}{r})^{\beta}\leq  C r^{\gamma} (\frac{|y|}{r})^{{\gamma}}\leq C|y|^{\gamma}.\] 
If $\beta\leq {\gamma}$,
$ C r^{\gamma} (\frac{|y|}{r})^{\beta}\leq C|y|^{\beta}.$

Finally, we let $\alpha=\min\{\beta,\gamma\}$ and finish the proof.
\end{proof}

\begin{theorem}\label{thm:c2alpha}
Suppose $c_\alpha\geq 0$, (A')(B)(C') hold and $\Omega$ satisfies the uniform exterior ball condition. 
Then there exists $u$ solves \eqref{convex operator} classically with a bounded function $g\in C^\kappa(\Omega^c)$ for some $\kappa>0$.
\end{theorem}
\begin{proof}
Using Theorem 5.7\cite{mou2018}, we obtain that \eqref{modified operator} admits a viscosity solution $u_q\in C(\mathbb R^d)$. Without loss of generality, we can assume that $u_q$ solves \eqref{modified operator} classically, see Theorem 5.3\cite{mou2017}. Then, using Theorem \ref{a priori estimate} and Lemma \ref{lem:bdry estimate}, we have there exists an $\alpha>0$ such that
\begin{equation*}
 \|u_q\|_{C^{2,\alpha}(\Omega_\delta))}+\|u_q\|_{C^\alpha(\mathbb R^d)}\leq C\left(\|g\|_{C^{\kappa}(\Omega^c)}+\sup_{\alpha\in\mathcal{A}}\|f_{\alpha}\|_{C^\alpha(\Omega)}\right)
\end{equation*}
where $C$ is independent of $q$. Using a diagonal argument, there exist a subsequence of $\{u_q\}_q$ and its limit $u\in C
_{\rm loc}^{2,\alpha}(\Omega)\cap C^\alpha(\mathbb R^d)$ such that $u$ solves \eqref{convex operator} classically.

\end{proof}

\begin{theorem}
Suppose the conditions in Theorem~\ref{thm:c2alpha} hold. Then \eqref{convex operator} admits a unique viscosity solution $u$, which is $C_{\rm loc}^{2,\alpha}(\Omega)\cap C^\alpha(\mathbb R^d)$ for some $\alpha>0$.
\end{theorem}

\begin{proof}
According to Theorem~\ref{thm:c2alpha}, \eqref{convex operator} admits a classical solution $u$.
Suppose there is another viscosity solution $v$($\not\equiv u$) of \eqref{convex operator}. Then $\max_{\mathbb R^d}|u-v|$ is obtained inside $\Omega$, i.e. there exists $x_0\in\Omega$ such that
\[|u-v|(x_0)=\max_{\mathbb R^d}|u-v|=\sigma>0.\]
We first assume that $u(x_0)>v(x_0)$.
By $u\in C_{\rm loc}^{2,\alpha}(\Omega)$, we can view $u$ as a test function and thus\[\sup_{\alpha\in\mathcal{A}}\{-\tr{ \, a_{\alpha}(x_0) D^2 u(x_0)}-\tilde{I}_{\alpha}[x_0,u]+b_{\alpha}(x_0)\cdot Du(x_0)+c_{\alpha}(x_0)v(x_0)+f_{\alpha}(x_0)\}\geq 0.\]
Then
\[-\sigma \inf_{\alpha\in\mathcal{A}} c_\alpha (x_0)=\tilde{\mathcal{I}}[x_0,u]-\sigma \inf_{\alpha\in\mathcal{A}} c_\alpha (x_0)\geq  \sup_{\alpha\in\mathcal{A}}\{-\tr{ \, a_{\alpha} D^2 u}-\tilde{I}_{\alpha}[\cdot,u]+b_{\alpha}\cdot Du+c_{\alpha}v+f_{\alpha}\}(x_0)\geq 0.\]
We get a contradiction if $\inf_\alpha c_\alpha(x_0)>0$, otherwise we do a  perturbation argument as follows. Take
\[C_3=\sup_{\alpha\in\mathcal{A},x\in\Omega}|b_\alpha(x)|.\]Recall Lemma 5.5\cite{mou2018}, there exists a function $\psi\in C^2(\Omega)\cap C(\mathbb{R}^d)$ such that
\begin{align*}
&    1\leq \psi\leq 2,\\
&    \mathcal{P}^+(D^2\psi)+\mathcal{P}^+_K(\psi)+C_3 |D\psi|\leq -1 \text{ in }\Omega.
\end{align*}

Now instead of considering $u$ as a test function, we choose $u-\epsilon\psi$ where $\epsilon<\sigma/4$. Then
\[u-\epsilon \psi\leq u-\epsilon= v-\epsilon \text{ in }\Omega^c \]
and
\[(u-\epsilon \psi)(x_0)\geq v(x_0)+\sigma-2\epsilon>v(x_0)+2\epsilon.\]
So we can assume $v-(u-\epsilon\psi)$ obtains its minimum at some $x_0'\in \Omega$ and thus
\[\tilde{\mathcal{I}}[x_0,u-\epsilon\psi]\geq 0\]
which is equivalent to
\begin{align*}
    \tilde{\mathcal{I}}[x_0,u-\epsilon\psi]&=\sup_{\alpha\in\mathcal{A}}\{-\tr{ \, a_{\alpha} D^2 (u-\epsilon\psi)}-\tilde{I}_{\alpha}[\cdot,(u-\epsilon\psi)]+b_{\alpha}\cdot D(u-\epsilon\psi)+c_{\alpha}(u-\epsilon\psi)+f_{\alpha}\}(x_0)\\
    &\leq \tilde{\mathcal{I}}[x_0,u]+\epsilon\sup_{\alpha\in\mathcal{A}}\{ \tr\,a_\alpha D^2\psi+\tilde{I}_\alpha[\cdot,\psi]+|b_\alpha||D\psi|- c_\alpha \psi\}(x_0)\\
    &\leq -\epsilon<0,
\end{align*}  
which leads to a contradiction. So $u(x_0)\leq v(x_0)$.
A similar argument shows that $u(x_0)\geq v(x_0)$.
In all we finish the proof.
\end{proof}

\end{document}